\documentclass{amsart}

\usepackage{amsmath, amssymb, amscd, mathrsfs, bm}

\usepackage{xcolor}
\usepackage[linktocpage=true]{hyperref}
\usepackage[only,mapsfrom]{stmaryrd}

\newtheorem{theorem}{Theorem}[section]
\newtheorem{proposition}[theorem]{Proposition}
\newtheorem{corollary}[theorem]{Corollary}

\theoremstyle{remark}
\newtheorem{remark}[theorem]{Remark}
\newtheorem{example}[theorem]{Example}
\newtheorem*{acknowledgments}{Acknowledgments}

\numberwithin{equation}{section}

\newcommand{\Order}{\mathcal{O}}
\newcommand{\into}{\hookrightarrow}

\newcommand{\onto}{\twoheadrightarrow}
\newcommand{\isomto}{\overset{\sim}{\to}}

\newcommand{\tensor}{\mathbin{\otimes}}
\newcommand{\closure}[1]{\overline{#1}}

\newcommand{\Z}{\mathbb{Z}}
\newcommand{\Q}{\mathbb{Q}}

\newcommand{\F}{\mathbb{F}}
\newcommand{\et}{\mathrm{et}}
\newcommand{\fppf}{\mathrm{fppf}}

\newcommand{\sep}{\mathrm{sep}}

\DeclareMathOperator{\Gal}{Gal}
\DeclareMathOperator{\Hom}{Hom}
\DeclareMathOperator{\End}{End}

\DeclareMathOperator{\Ker}{Ker}

\DeclareMathOperator{\Ext}{Ext}

\DeclareMathOperator{\Spec}{Spec}
\DeclareMathOperator{\Ab}{Ab}

\DeclareMathOperator{\Sel}{Sel}

\DeclareFontFamily{U}{wncy}{}
\DeclareFontShape{U}{wncy}{m}{n}{<->wncyr10}{}
\DeclareSymbolFont{mcy}{U}{wncy}{m}{n}
\DeclareMathSymbol{\Sha}{\mathord}{mcy}{"58}

\makeatletter
\@namedef{subjclassname@2020}{\textup{2020} Mathematics Subject Classification}
\makeatother

\title[(Co)finiteness of fine Selmer groups over function fields]
{Finiteness and cofiniteness of fine Selmer groups over function fields}

\author{Sohan Ghosh}
\address{
	Harish Chandra Research Institute, A CI of Homi Bhabha National
	Institute, Chhatnag Road, Jhunsi, Prayagraj (Allahabad) 211 019 India
}
\email{ghoshsohan4@gmail.com}

\author{Jishnu Ray}
\address{
	Harish Chandra Research Institute, A CI of Homi Bhabha National
	Institute, Chhatnag Road, Jhunsi, Prayagraj (Allahabad) 211 019 India
}
\email{jishnuray@hri.res.in; jishnuray1992@gmail.com}

\author{Takashi Suzuki}
\address{
	Department of Mathematics, Chuo University,
	1-13-27 Kasuga, Bunkyo-ku, Tokyo 112-8551, Japan
}
\email{tsuzuki@gug.math.chuo-u.ac.jp}

\date{April 5, 2025}

\begin{document}

\begin{abstract}
	We prove that the dual fine Selmer group of an abelian variety
	over the unramified $\Z_{p}$-extension of a function field
	is finitely generated over $\Z_{p}$.
	This is a function field version of a conjecture of Coates--Sujatha.
	We further prove that the fine Selmer group is finite (respectively zero)
	if the separable $p$-primary torsion of the abelian variety is finite
	(respectively zero).
	These results are then generalized to
	certain ramified $p$-adic Lie extensions.
\end{abstract}

\maketitle

\tableofcontents


\section{Introduction}
\label{0049}

Coates--Sujatha \cite{CS05} introduce a certain subgroup
of the Selmer group of an elliptic curve
over the cyclotomic $\Z_{p}$-extension of a number field,
which they call the fine Selmer group.
In \cite[Section 3, Conjecture A]{CS05},
they conjecture that its Pontryagin dual is
always finitely generated over $\Z_{p}$
(when the corresponding statement for the full Selmer group
is known to be false in general).
In this paper, we will prove a function field version of this conjecture
for all abelian varieties.%
\footnote{
    \cite{CS05} assumes $p \ne 2$ for simplicity
    while everything is expected to be extended to $p = 2$.
    On the other hand, our results and methods are uniform in all $p$.
    The difference comes from the usual irregular behavior of the structure of
    $\Gal(\Q(\zeta_{p^{\infty}}) / \Q) \cong \Z_{p}^{\times}$ for $p = 2$
    compared $\Gal(\F_{p^{p^{\infty}}} / \F_{p}) \cong \Z_{p}$
    and from the presence/absence of archimedean places
    contributing to $2$-torsion.
}
An earlier result in this direction is
the case of ordinary elliptic curves by
Ghosh--Jha--Shekhar \cite[Theorem 3.7]{GJS22}.
Our proof is based on a different method,
where systematic treatment of schematic closures
of separable $p$-torsion of abelian varieties plays a key role.

Let $X$ be a proper smooth geometrically connected curve
over a finite field $k$ of characteristic $p$
with function field $K$.
Let $A$ be an abelian variety over $K$.
Let $k_{\infty}$ be the $\Z_{p}$-extension of $k$
and $K_{\infty} = K k_{\infty}$ the unramified $\Z_{p}$-extension of $K$.
Let $S$ be any finite set of places of $K$
containing all places of bad reduction for $A$,
so that $A$ extends to an abelian scheme $\mathcal{A}$
over $U := X \setminus S$.
Let $U_{0}$ be the set of closed points of $U$.
Let $S_{\infty}$, $U_{\infty}$, $U_{\infty, 0}$ be
the inverse images of $S$, $U$ and $U_{0}$
in $X_{\infty} = X \times_{k} k_{\infty}$, respectively.
Define the $p$-primary $S$-fine Selmer group
$R^{S}(A / K_{\infty})$ of $A$ over $K_{\infty}$
as the kernel of the map on the fppf cohomology groups
	\[
                H^{1}(K_{\infty}, A[p^{\infty}])
            \to
                    \bigoplus_{w \in S_{\infty}}
                        H^{1}(K_{\infty, w}, A[p^{\infty}])
                \oplus
                    \bigoplus_{w \in U_{\infty, 0}}
                        H^{1}(K_{\infty, w}, A)[p^{\infty}],
	\]
which is a $\Lambda = \Z_{p}[[\Gal(K_{\infty} / K)]]$-module.
If $S = \emptyset$,%
\footnote{
	The above definition of $R^{S}(A / K_{\infty})$ makes sense
	for any finite $S$ including $S = \emptyset$,
	but we treat $R^{S}(A / K_{\infty})$ below
	only when $S$ contains all bad places for $A$.
}
it is the usual $p$-primary Selmer group $\Sel(A / K_{\infty})$,
whose Pontryagin dual is finitely generated torsion over
$\Lambda$ by Ochiai--Trihan \cite[Theorem 1.7]{OT09}.
In general, $R^{S}(A / K_{\infty})$ is a submodule of $\Sel(A / K_{\infty})$.
For any finite set of places $S'$ of $K$ containing $S$,
we have $R^{S'}(A / K_{\infty}) \subset R^{S}(A / K_{\infty})$.
Let $R^{S}(A / K_{\infty})^{\vee}$ be
the Pontryagin dual of $R^{S}(A / K_{\infty})$.

Let $A[p]^{\natural}$ be the $\Gal(K^{\sep} / K)$-module
$A(K^{\sep})[p]$ viewed as a finite \'etale group scheme over $K$.
Let $\mathcal{A}[p]^{\natural}$ be
the schematic closure of $A[p]^{\natural}$ in $\mathcal{A}[p]$.
It is a finite flat group scheme over $U$
with (\'etale) generic fiber $A[p]^{\natural}$.
In this paper, we will prove the following:

\begin{theorem} \label{0014}
	Assume that $S \ne \emptyset$ and
	$\mathcal{A}[p]^{\natural}$ is \'etale over $U$.
	\begin{enumerate}
		\item \label{0015}
			$R^{S}(A / K_{\infty})$ is of cofinite type over $\Z_{p}$
			(namely, its part killed by $p$ is finite).
		\item \label{0016}
			For any finite set of places $S'$ of $K$ containing $S$,
			we have $R^{S'}(A / K_{\infty}) = R^{S}(A / K_{\infty})$.
	\end{enumerate}
\end{theorem}

Statement \eqref{0015} means that
the Iwasawa $\mu$-invariant of $R^{S}(A / K_{\infty})^{\vee}$ is zero.
It gives a function field version of \cite[Section 3, Conjecture A]{CS05}
and is a generalization of \cite[Theorem 3.7]{GJS22}.
The assumptions in Theorem \ref{0014} are satisfied if $S$ is large enough.
The dependence of $R^{S}(A / K_{\infty})$ on $S$ is a subtle matter
pointed out in \cite[Remark 3.2]{GJS22}
that does not arise in the number field setting
(see Wuthrich's remark after \cite[Lemma 2.3]{Wut07}).
In Example \ref{0044}, we will also see these subtleties
by showing that the \'etaleness assumption
on $\mathcal{A}[p]^{\natural}$ is necessary.

Finer information about $R^{S}(A / K_{\infty})$ can be obtained
if the separable $p$-primary torsion $A(K^{\sep})[p^{\infty}]$ is small:

\begin{theorem} \label{0039}
	Assume that $S \ne \emptyset$.
	Let $n \ge 0$ be an integer.
	If $A(K^{\sep})[p^{\infty}]$ is killed by $p^{n}$,
	then $R^{S}(A / K_{\infty})$ is also killed by $p^{n}$.
\end{theorem}

\begin{corollary} \label{0040}
	Assume that $S \ne \emptyset$.
	\begin{enumerate}
		\item \label{0018}
			If $A(K^{\sep})[p^{\infty}]$ is finite,
			then $R^{S}(A / K_{\infty})$ has bounded torsion
			(killed by a fixed power of $p$).
		\item \label{0041}
			If $A(K^{\sep})[p^{\infty}]$ is finite and
			$\mathcal{A}[p]^{\natural}$ is \'etale over $U$,
			then $R^{S}(A / K_{\infty})$ is finite.
		\item \label{0042}
			If $A(K^{\sep})[p] = 0$, then $R^{S}(A / K_{\infty}) = 0$.
	\end{enumerate}
\end{corollary}

Statement \eqref{0018} means that
the Iwasawa $\lambda$-invariant of $R^{S}(A / K_{\infty})^{\vee}$ is zero in this case.
The finiteness of $A(K^{\sep})[p^{\infty}]$ holds
if $\End_{\closure{K}}(A) \cong \Z$ by R\"ossler \cite[Theorem 1.4]{Ros13}
and in particular if $A$ is a non-isotrivial elliptic curve.
See R\"ossler \cite[Theorem 1.2]{Ros20}
and D'Addezio \cite[Theorem 1.1.3]{Dad23}
for generalizations and variants.
The vanishing of $A(K^{\sep})[p]$ holds
if $A$ is ordinary and has maximal Kodaira--Spencer rank
by Voloch \cite[Section 4, Proposition]{Vol95}
or if $A[p]$ is infinitesimal.
These results (in particular for non-isotrivial elliptic curves) show
unexpected limitation of $R^{S}(A / K_{\infty})$
for Iwasawa-theoretic purposes,
since $\Lambda$-modules of finite order have trivial characteristic ideals.

On the other hand, $R^{S}(A / K_{\infty})$ does not always have bounded torsion
as we will see in Example \ref{0043}.
Also, there exists an ordinary abelian variety $A$
over some function field $K$
with no isotrivial factors such that $A(K^{\sep})[p^{\infty}]$ is infinite
by Helm \cite[Theorem 1.1]{Hel22}.
It is an interesting problem to calculate $R^{S}(A / K_{\infty})$
and its characteristic ideal for Helm's example.

We can actually describe $R^{S}(A / K_{\infty})$
(or its version with $p^{n}$-torsion)
as an fppf cohomology group of $X$ with coefficients in
a certain quasi-finite flat separated group scheme over $X$
with \'etale generic fiber when $S \ne \emptyset$
(Proposition \ref{0008}).
Theorems \ref{0014} and \ref{0039} basically follow from this description.
As the statements of Theorems \ref{0014} and \ref{0039} make
no reference to the Galois actions,
they are essentially statements
over arbitrary algebraically closed base fields $k$ of characteristic $p > 0$.
Also, the statements make sense
for any $p$-divisible group over $U$
in place of $\mathcal{A}[p^{\infty}]$.
We will work in this generality.

As suggested by one of the referees,
we can actually generalize Theorems \ref{0014} and \ref{0039}
from $K_{\infty}$ to certain (not necessarily commutative or unramified)
$p$-adic Lie extensions of $K$.
We will give this generalization in
Theorems \ref{0048} and \ref{0050}.
The proof uses the original Theorems \ref{0014} and \ref{0039}
(or more precisely, Propositions \ref{0028} and \ref{0012})
in an essential manner.
This is a fine Selmer version of \cite[Theorem 1.9]{OT09}
and a function field version of \cite[Lemma 3.2, Corollary 3.3]{CS05}.

We make one note on the literature.
The idea that $A(K^{\sep})[p^{\infty}]$ governs Hasse principles
(such as Tate--Shafarevich groups)
for function fields already appears in the work of
Gonz\'alez-Avil\'es--Tan \cite{GAT12} (among others).
Their main theorem and Proposition 3.2 are closely related to
our results.
However, their setting over $K$ itself and our setting over $K_{\infty}$
have significant differences.
For example, the group of locally trivial $\Z / p \Z$-torsors over $K$ is trivial
by the Chebotarev density theorem.
But the group of locally trivial $\Z / p \Z$-torsors over $K_{\infty}$ is
the $\Gal(\closure{k} / k_{\infty})$-fixed part of
$H^{1}(X \times_{k} \closure{k}, \Z / p \Z)$,
which is non-trivial if, for example,
$X$ is an elliptic curve with $\pi_{0}(X[p]) \cong \Z / p \Z$.
Our Example \ref{0043} shows that a direct analogue of
\cite[Proposition 3.2]{GAT12} for $K_{\infty}$ does not hold.
It is $K_{\infty}$ that matters in Iwasawa theory.

In the rest of the paper,
all groups and group schemes are assumed to be commutative
(except the Galois group in Section \ref{0051}).

\begin{acknowledgments}
	The authors thank Marco D'Addezio, Somnath Jha, Damian R\"ossler,
	Ki-Seng Tan and Fabien Trihan for helpful discussions and references.
	They are also grateful to the referees
	for the careful reading and helpful comments,
	especially about the suggestions of Theorems \ref{0048} and \ref{0050}.
	The first author acknowledges the support received
	from the HRI postdoctoral fellowship.
	The second author also gratefully acknowledges
	support from the Inspire research grant,
	DST, Govt.\ of India.
\end{acknowledgments}


\section{Some facts about finite flat group schemes}

We will collect here some general facts about finite flat group schemes.

For a group scheme $G$ locally of finite type over a field $K$,
let $G^{\natural} \subset G$ be the maximal smooth closed subgroup
(\cite[Lemma C.4.1]{CGP15}).
The formation $G \mapsto G^{\natural}$ commutes with
(possibly transcendental) separable base field extensions
(\cite[Tag 030O]{Sta24}) by \cite[Lemma C.4.1]{CGP15}.
Set $G^{/ \natural} = G / G^{\natural}$.
Then $(G^{/ \natural})^{\natural} = 0$.
If $G$ is finite, then $G^{\natural}$ is
the $\Gal(K^{\sep} / K)$-module $G(K^{\sep})$
viewed as an \'etale group scheme over $K$.

\begin{proposition} \label{0002}
	Let $K' / K$ be a separable extension of fields.
	Let $N$ be a finite group scheme over $K$ with $N^{\natural} = 0$.
	Then the natural map
	$H^{1}(K, N) \to H^{1}(K', N)$
	is injective.
\end{proposition}

\begin{proof}
	This is essentially proved during the proof of \cite[Proposition 4.10]{DH19}.
	We recall the proof here.
	Let $Z$ be an fppf $N$-torsor over $K$ with $Z(K') \ne \emptyset$.
	Then $Z$ is finite over $K$,
	so $Z(K'') \ne \emptyset$ for some finite subextension $K''$ of $K'$.
	Since $K' / K$ is separable, so is $K'' / K$.
	Hence the class of $Z$ belongs to the subgroup
	$H^{1}(K_{\et}, N)$ of $H^{1}(K, N)$.
	But $N$ is zero as a sheaf on $\Spec K_{\et}$ since $N^{\natural} = 0$.
	Therefore $H^{1}(K_{\et}, N) = 0$,
	so $Z$ is trivial.
\end{proof}

The following proposition allows us to modify
a finite flat scheme into an \'etale scheme:

\begin{proposition} \label{0031}
	Let $R$ be a discrete valuation ring with fraction field $K$.
	Let $j \colon \Spec K_{\et} \to \Spec R_{\et}$ be
	the natural morphism on the small \'etale sites.
	Let $V$ be a finite flat scheme over $R$
	with \'etale generic fiber $V_{K}$.
	\begin{enumerate}
		\item \label{0033}
			There exists a unique morphism
			$j_{\ast} V_{K} \to V$ of $R$-schemes
			that extends the identity morphism
			$V_{K} \to V_{K}$ over $K$.
		\item \label{0034}
			The morphism $j_{\ast} V_{K} \to V$ is an isomorphism
			if $V$ is \'etale over $R$.
		\item \label{0035}
			If moreover $V$ is a group scheme over $R$,
			then $j_{\ast} V_{K} \to V$ is a group scheme morphism.
	\end{enumerate}
\end{proposition}

Here $V_{K}$ is viewed as a representable sheaf of sets on $\Spec K_{\et}$
and the direct image sheaf $j_{\ast} V_{K}$ on $\Spec R_{\et}$ is
representable (see the proof below),
so that it makes sense to talk about an $R$-scheme morphism $j_{\ast} V_{K} \to V$.

\begin{proof}
	Let $V'$ be the normalization (\cite[Tag 035H]{Sta24}) of $\Spec R$ in $V_{K}$.
	Then $j_{\ast} V_{K}$ is
	the maximal open subscheme of $V'$ \'etale over $R$,
	since for any \'etale $R$-scheme $Z$,
	a morphism $Z_{K} \to V_{K}$ on the generic fibers
	uniquely comes from a morphism $Z \to V'$ over $R$
	(by \cite[Tag 035L]{Sta24} and $Z$ being normal).
	Since $V$ is finite flat over $R$,
	the inclusion $V_{K} \into V$ uniquely factors through $V'$
	by \cite[Tag 035I]{Sta24}.
	The composite $j_{\ast} V_{K} \into V' \to V$ is
	the desired unique extension.
	The final statement about group scheme morphisms follows from
	the fact that $j_{\ast}$ commutes with fiber products
	(being the right adjoint)
	and the uniqueness property just proved.
\end{proof}


\section{Fine Selmer groups of finite flat group schemes}
\label{0053}

Let $k$ be an algebraically closed field of characteristic $p > 0$.
Let $X$ be a proper smooth connected curve over $k$ with function field $K$.
Let $X_{0}$ be the set of closed points of $X$.
For $v \in X_{0}$, let $\Order_{v}$ be
the completed local ring of $X$ at $v$
and $K_{v}$ its fraction field.
Let $U \subset X$ be a dense open subscheme
with complement $S$.
Let $U_{0}$ be the set of closed points of $U$.
Let $N$ be a finite flat group scheme over $U$
with generic fiber $N_{K}$.

Let
	\begin{equation} \label{0022}
		\begin{split}
					R^{S}(N)
			&	=
					\Ker \left(
							H^{1}(K, N)
						\to
								\bigoplus_{v \in S}
									H^{1}(K_{v}, N)
							\oplus
								\bigoplus_{v \in U_{0}}
									H^{2}_{v}(\Order_{v}, N)
					\right)
			\\
			&	\cong
					\Ker \left(
							H^{1}(U, N)
						\to
							\bigoplus_{v \in S}
								H^{1}(K_{v}, N)
					\right)
		\end{split}
	\end{equation}
be the $S$-fine Selmer group of $N$,
where $H^{2}_{v}$ denotes the fppf cohomology with support on the closed point
(\cite[Chapter III, Section 0]{Mil06}).
Note that $H^{2}_{v}(\Order_{v}, N)$ does not change
if $\Order_{v}$ is replaced by the henselian local ring at $v$;
see \cite[Proposition 2.6.2]{Suz20a} or \cite[Lemma 2.6]{DH19}.
Note also that $H^{i}_{v}(\Order_{v}, N) = 0$ for $i \le 1$
and $H^{2}_{v}(\Order_{v}, N) \cong H^{1}(K_{v}, N) / H^{1}(\Order_{v}, N)$
by the same proof as \cite[Chapter III, Lemma 1.1]{Mil06}.
For a smaller dense open subscheme $U' \subset U$ with $S' = X \setminus U'$,
the above construction applied to $N \times_{U} U'$ defines
another group $R^{S'}(N)$.
We have a natural injection
	\begin{equation} \label{0021}
		R^{S'}(N) \into R^{S}(N).
	\end{equation}
We note here that $R^{S}(N)$ is denoted by $D^{1}(U, N)$
in \cite[the paragraph after Chapter II, Corollary 3.3]{Mil06}
and \cite[Section 3]{DH19}.

Below we are going to describe $R^{S}(N)$ as a certain cohomology group of $X$.
For this, we will modify $N$ and extend it to the whole $X$
in the following particular manner.
Let $N^{\natural} \subset N$
be the schematic closure of
$N_{K}^{\natural}$ ($= (N_{K})^{\natural}$) in $N$,
which is a finite flat closed group subscheme over $U$
with (\'etale) generic fiber $N_{K}^{\natural}$.
Let $N^{\natural, S}$ be the quasi-finite flat separated group scheme over $X$
uniquely determined by the following requirements:
\begin{enumerate}
	\item
		The restriction of $N^{\natural, S}$ to $U$ is $N^{\natural}$.
	\item
		For any $v \in S$,
		the restriction of $N^{\natural, S}$ to $\Spec \Order_{v}$ is
		$j_{v, \ast}(N_{K}^{\natural})$,
		where $j_{v} \colon \Spec K_{v, \et} \to \Spec \Order_{v, \et}$ is
		the natural morphism on the small \'etale sites.
\end{enumerate}

For a smaller dense open subscheme $U' \subset U$ with $S' = X \setminus U'$,
the above construction applied to $N \times_{U} U'$ defines
another group scheme $N^{\natural, S'}$ over $X$.
By Proposition \ref{0031},
we have a group scheme morphism
	\begin{equation} \label{0020}
		N^{\natural, S'} \to N^{\natural, S}
	\end{equation}
over $\Order_{v'}$ for each $v' \in S' \setminus S$
and hence over the whole $X$.

Let $g \colon \Spec K_{\et} \to X_{\et}$ be
the natural morphism on the small \'etale sites.
Then we similarly have a morphism
	\begin{equation} \label{0032}
			g_{\ast} N_{K}^{\natural}
		\to
			N^{\natural, S}
	\end{equation}
compatible with \eqref{0020}.
The morphisms \eqref{0020} and \eqref{0032} are
the unique extensions of the identity morphism
$N_{K}^{\natural} \to N_{K}^{\natural}$.

With these constructions, we can describe $R^{S}(N)$:

\begin{proposition} \label{0008}
	Assume that $S \ne \emptyset$.
	Then there exists a canonical isomorphism
	$R^{S}(N) \cong H^{1}(X, N^{\natural, S})$
	compatible with \eqref{0021} and \eqref{0020}.
\end{proposition}

\begin{proof}
	Set $N^{/ \natural} = N / N^{\natural}$,
	which is a finite flat group scheme over $U$.
	The exact sequence
	$0 \to N^{\natural} \to N \to N^{/ \natural} \to 0$
	induces exact sequences
		\begin{gather*}
					0
				\to
					H^{1}(K, N^{\natural})
				\to
					H^{1}(K, N)
				\to
					H^{1}(K, N^{/ \natural}),
			\\
					0
				\to
					H^{1}(K_{v}, N^{\natural})
				\to
					H^{1}(K_{v}, N)
				\to
					H^{1}(K_{v}, N^{/ \natural}).
		\end{gather*}
	For any $v \in U_{0}$,
	the same exact sequence induces an exact sequence
		\[
				0
			\to
				H^{2}_{v}(\Order_{v}, N^{\natural})
			\to
				H^{2}_{v}(\Order_{v}, N)
			\to
				H^{2}_{v}(\Order_{v}, N^{/ \natural}).
		\]
	Therefore we have an exact sequence
		\[
				0
			\to
				R^{S}(N^{\natural})
			\to
				R^{S}(N)
			\to
				R^{S}(N^{/ \natural}).
		\]
	But, since the extension $K_{v} / K$ is separable,
        the map $H^{1}(K, N^{/ \natural}) \to H^{1}(K_{v}, N^{/ \natural})$
	is injective for any $v \in S$ ($\ne \emptyset$)
	by Proposition \ref{0002}.
	Hence $R^{S}(N^{/ \natural}) = 0$,
	so $R^{S}(N^{\natural}) \isomto R^{S}(N)$.
	
	On the other hand, we have
	$H^{0}(\Order_{v}, N^{\natural, S}) \cong H^{0}(K_{v}, N^{\natural})$
	and $H^{i}(\Order_{v}, N^{\natural, S}) = 0$
	for $v \in S$ and $i \ge 1$
	since $k$ is algebraically closed and thus $\Order_{v}$ is strictly henselian.
	Hence the long exact sequence
		\[
				\cdots
			\to
				\bigoplus_{v \in X_{0}}
					H^{i}_{v}(\Order_{v}, N^{\natural, S})
			\to
				H^{i}(X, N^{\natural, S})
			\to
				H^{i}(K, N^{\natural})
			\to
				\cdots
		\]
	(recall that the generic fiber of $N^{\natural, S}$ is
	$N^{\natural}$ or, more precisely, $N_{K}^{\natural}$)
	shows that $H^{1}(X, N^{\natural, S})$ is
	isomorphic to $R^{S}(N^{\natural})$.
	The compatibility with \eqref{0021} and \eqref{0020}
	follows from the construction.
\end{proof}

\begin{proposition} \label{0011}
	Assume that $S \ne \emptyset$ and $N^{\natural}$ is \'etale over $U$.
	Then \eqref{0032} is an isomorphism
	and \eqref{0021} and \eqref{0020} are isomorphisms for any finite $S' \supset S$.
	In particular,
	$N^{\natural, S} \cong g_{\ast} N_{K}^{\natural}$
	is \'etale over $X$
	and $R^{S}(N) \cong H^{1}(X, g_{\ast} N_{K}^{\natural})$ is finite.
\end{proposition}

\begin{proof}
	The statements about \eqref{0032}, \eqref{0021} and \eqref{0020}
	follow from Proposition \ref{0031} \eqref{0034}.
	Since $g_{\ast} N_{K}^{\natural}$ is a constructible sheaf on $X_{\et}$,
	the group $H^{1}(X, g_{\ast} N_{K}^{\natural})$ is finite
	by \cite[Chapter VI, Theorem 2.1]{Mil80}.
\end{proof}


\section{%
	\texorpdfstring{Fine Selmer groups of $p$-divisible groups and abelian varieties}
	{Fine Selmer groups of p-divisible groups and abelian varieties}
}
\label{0025}

Let $U \subset X$ and $S = X \setminus U$ as before.
Let $G$ be a $p$-divisible group over $U$ with generic fiber $G_{K}$.
For $n \ge 0$, we denote $R^{S}(G[p^{n}])$ as $R^{S}_{n}(G)$.
The same formula as \eqref{0022} defines
the fine Selmer group $R^{S}(G)$ of $G$,
which is the direct limit of $R^{S}_{n}(G)$ in $n$.
Let $G^{\natural}$, $G^{\natural, S}$ and $G_{K}^{\natural}$ be
the direct limits of $G[p^{n}]^{\natural}$, $G[p^{n}]^{\natural, S}$
and $G_{K}[p^{n}]^{\natural}$, respectively,
in $\Ab(U_{\fppf})$, $\Ab(X_{\fppf})$
and $\Ab(K_{\fppf})$ (or in $\Ab(K_{\et})$), respectively.
For any finite $S' \supset S$, we have natural morphisms
	\begin{equation} \label{0027}
		\begin{gathered}
					g_{\ast} G_{K}^{\natural}
				\to
					G^{\natural, S'}
				\to
					G^{\natural, S},
			\\
				R^{S'}(G) \into R^{S}(G).
		\end{gathered}
	\end{equation}

The next two propositions (Propositions \ref{0007} and \ref{0009})
allow us to control $R^{S}(G)$ by $R^{S}_{1}(G)$.

\begin{proposition} \label{0007}
	The kernel and the cokernel of the natural map
	$R^{S}_{1}(G) \to R^{S}(G)[p]$
	are finite.
\end{proposition}

\begin{proof}
	We have exact sequences
		\begin{gather*}
				H^{1}(K, G[p])
			\to
				H^{1}(K, G)
			\stackrel{p}{\to}
				H^{1}(K, G),
            \\
				H^{1}(K_{v}, G[p])
			\to
				H^{1}(K_{v}, G)
			\stackrel{p}{\to}
				H^{1}(K_{v}, G),
		\end{gather*}
	where the kernel of the first map in each sequence is finite.
	For any $v \in U_{0}$,
	we have an exact sequence
		\[
				0
			\to
				H^{2}_{v}(\Order_{v}, G[p])
			\to
				H^{2}_{v}(\Order_{v}, G)
			\stackrel{p}{\to}
				H^{2}_{v}(\Order_{v}, G).
		\]
	Therefore the sequence
		\[
				0
			\to
				R^{S}_{1}(G)
			\to
				R^{S}(G)
			\stackrel{p}{\to}
				R^{S}(G)
		\]
	is exact up to finite groups,
	which proves the proposition.
\end{proof}

\begin{proposition} \label{0009}
	If $G[p]^{\natural}$ is \'etale over $U$,
	then $G[p^{n}]^{\natural}$ and $G^{\natural}$ are \'etale over $U$
	and $G[p^{n}]^{\natural, S}$ and $G^{\natural, S}$ are \'etale over $X$
	for all $n$.
\end{proposition}

\begin{proof}
	It is enough to treat $G[p^{n}]^{\natural}$.
	Assume that $G[p^{n}]^{\natural}$ is \'etale
	over $U$ for some $n \ge 1$.
	We will prove the same for $G[p^{n + 1}]^{\natural}$.
	We have a natural sequence
            \[
                    0
                \to
                    G[p^{n}]^{\natural}
                \to
                    G[p^{n + 1}]^{\natural}
                \to
                    G[p]^{\natural}
            \]
	over $U$,
	whose generic fiber is exact.
	Since $G[p]^{\natural}$ is \'etale
	and $G[p^{n + 1}]^{\natural}$ is finite flat,
	it follows that the schematic image $G[p]^{\natural \prime}$ of the morphism
	$G[p^{n + 1}]^{\natural} \to G[p]^{\natural}$
	is \'etale and the induced morphism
	$G[p^{n + 1}]^{\natural} \to G[p]^{\natural \prime}$
	is faithfully flat.
	In particular, $\Ker(G[p^{n + 1}]^{\natural} \onto G[p]^{\natural \prime})$
	is finite flat over $U$.
	It follows that
	$\Ker(G[p^{n + 1}]^{\natural} \onto G[p]^{\natural \prime}) = G[p^{n}]^{\natural}$,
	which is \'etale over $U$.
	Therefore $G[p^{n + 1}]^{\natural}$ is \'etale over $U$.
	Induction then gives the result.
\end{proof}

Now we can describe $R^{S}(G)$:

\begin{proposition} \label{0028}
	Assume $S \ne \emptyset$.
	\begin{enumerate}
		\item \label{0036}
			There exists a canonical isomorphism
			$R^{S}(G) \cong H^{1}(X, G^{\natural, S})$.
		\item \label{0037}
			Assume moreover that $G[p]^{\natural}$ is \'etale over $U$.
			Then the morphisms \eqref{0027} are isomorphisms
			for any finite $S' \supset S$.
			The group $G^{\natural, S} \cong g_{\ast} G_{K}^{\natural}$ is \'etale over $X$
			and $R^{S}(G) \cong H^{1}(X, g_{\ast} G_{K}^{\natural})$ is of cofinite type.
	\end{enumerate}
\end{proposition}

\begin{proof}
	Apply Propositions \ref{0008} and \ref{0011}
	to each $G[p^{n}]$
	using Propositions \ref{0007} and \ref{0009}.
\end{proof}

\begin{proposition} \label{0012}
	Let $n \ge 0$.
	Assume that $G(K^{\sep})$ is killed by $p^{n}$.
	Assume also that $S \ne \emptyset$.
	Then $R^{S}_{n}(G) \isomto R^{S}(G)$,
	which is killed by $p^{n}$.
	If furthermore $G[p]^{\natural}$ is \'etale over $U$,
	then $R^{S}(G)$ is finite.
\end{proposition}

\begin{proof}
	By assumption, we have
	$G[p^{n}]^{\natural} = G[p^{n + 1}]^{\natural} = \cdots$.
	Hence $G^{\natural, S} = G[p^{n}]^{\natural, S}$.
		Therefore the result follows from
	Propositions \ref{0008} and \ref{0011}.
\end{proof}

This proposition for $n = 0$ means that
if $G(K^{\sep}) = 0$ and $S \ne \emptyset$,
then $R^{S}(G) = 0$.

Let $A$ be an abelian variety over $K$.
Let $S$ be a finite set of places of $K$
containing all places of bad reduction for $A$.
Set $U = X \setminus S$,
so that $A$ extends to an abelian scheme $\mathcal{A}$ over $U$.
Applying the constructions of this section
to $G = \mathcal{A}[p^{\infty}]$,
we have the $S$-fine Selmer groups
$R^{S}_{n}(A / K) := R^{S}_{n}(\mathcal{A}[p^{\infty}])$ and
$R^{S}(A / K) := R^{S}(\mathcal{A}[p^{\infty}])$ of $A$.
These definitions are consistent with
the definition of $R^{S}(A / K)$ (or $R^{S}(A / K_{\infty})$) in Section \ref{0049},
since for any $v \in U_{0}$, we have
$H^{1}_{v}(\Order_{v}, \mathcal{A}) = 0$,
$H^{1}(K_{v}, A) \isomto H^{2}_{v}(\Order_{v}, \mathcal{A})$
and so
	\[
			H^{2}_{v}(\Order_{v}, \mathcal{A}[p^{n}])
		\isomto
			H^{2}_{v}(\Order_{v}, \mathcal{A})[p^{n}]
		\cong
			H^{1}(K_{v}, A)[p^{n}].
	\]
Hence the results in this and the previous sections can be interpreted
as results on $R^{S}_{n}(A / K)$ and $R^{S}(A / K)$.


\section{Connection to the finite base field case}
\label{0024}

Let $X$ be a proper smooth geometrically connected curve
over a finite field $k$ of characteristic $p$.
Let $k_{\infty}$ be the $\Z_{p}$-extension of $k$.
Let $K_{\infty} = K k_{\infty}$ be the unramified $\Z_{p}$-extension of $K$.
Let $S$ be a finite set of closed points of $X$
and set $U = X \setminus S$.
Let $G$ be a $p$-divisible group over $U$.
We have defined the fine Selmer groups
$R^{S}_{n}(G \times_{k} \closure{k})$
and $R^{S}(G \times_{k} \closure{k})$
of $G \times_{k} \closure{k}$ over $U \times_{k} \closure{k}$
in Section \ref{0025}
(where the upper script $S$ is more precisely
the inverse image of $S$ in $X \times_{k} \closure{k}$).
They have natural $\Gal(\closure{k} / k) \cong \Gal(K \closure{k} / K)$-actions.
The fine Selmer group
$R^{S}(G \times_{k} k_{\infty})$ is defined similarly,
but it is a $\Gal(K_{\infty} / K)$-module.

\begin{proposition} \label{0052}
	We have
		\[
				R^{S}_{n}(G \times_{k} \closure{k})^{\Gal(\closure{k} / k_{\infty})}
			\cong
				R^{S}_{n}(G \times_{k} k_{\infty})
		\]
	as $\Gal(K_{\infty} / K)$-modules.
	The same holds with $R^{S}_{n}$ replaced by $R^{S}$.
\end{proposition}

\begin{proof}
	Since these fine Selmer groups are torsion and $p$-primary
	and the group $\Gal(\closure{k} / k_{\infty})$ is pro-prime-to-$p$,
	the proposition follows from the Hochschild--Serre spectral sequence
	for fppf cohomology groups \cite[Chapter III, Remark 2.21 (a), (b)]{Mil80}.
\end{proof}

Therefore all the cofiniteness, finiteness and vanishing results
in the previous sections apply to $R^{S}_{n}(G \times_{k} k_{\infty})$.
For $G = \mathcal{A}[p^{\infty}]$,
Propositions \ref{0028} and \ref{0012} prove
Theorems \ref{0014} and \ref{0039}, respectively.

\begin{remark}
	The isomorphism in Proposition \ref{0008} over $\closure{k}$
	commutes with the action of $\Gal(\closure{k} / k)$.
	Hence we have isomorphisms
		\[
				R^{S}_{n}(G \times_{k} \closure{k})
			\cong
				H^{1}(X \times_{k} \closure{k}, G[p^{n}]^{\natural, S}),
		\]
		\[
				R^{S}_{n}(G \times_{k} k_{\infty})
			\cong
				H^{1}(X \times_{k} k_{\infty}, G[p^{n}]^{\natural, S})
		\]
	as $\Gal(\closure{k} / k)$-modules
	and $\Gal(k_{\infty} / k)$-modules, respectively.
\end{remark}


\section{Examples and a remark}

\begin{example} \label{0044}
	We show here that the \'etaleness assumption of $\mathcal{A}[p]^{\natural}$
	in Theorem \ref{0014} is necessary
	and hence $S$ needs to contain more than just the bad places for $A$.
	
	Let $A'$ be any ordinary elliptic curve over $K$
	with places of good supersingular reduction.
	For example, we can take $K = k(t)$ of characteristic $3$ and
	$A' \colon y^{2} = x^{3} + t x^{2} - x$,
	which is non-isotrivial and hence ordinary
	and has $t = 0$ as a supersingular place.
	Let $A = A'^{(p)}$ be the Frobenius twist of $A'$,
	which is ordinary with places of good supersingular reduction
	(at the same places where $A'$ has good supersingular reduction).
	Let $S$ be the set of bad places for $A$ (or $A'$).
	We will show that $R^{S}(A / K_{\infty})$ is not of cofinite type.
	
	Replacing $k$ by $\closure{k}$ as in Section \ref{0024},
	it is enough to show that $R^{S}(A / K)$ is not of cofinite type
	when $k = \closure{k}$.
	Let $U = X \setminus S$,
	so that $A$ extends to an abelian scheme $\mathcal{A}$ over $U$.
	Then $A[p]^{\natural}$ is given by
	the kernel of the Verschiebung $V \colon A \to A'$.
	Its schematic closure $\mathcal{A}[p]^{\natural}$ in $\mathcal{A}[p]$
	has non-\'etale fiber at a supersingular place of $A$.
	By Propositions \ref{0028} \eqref{0036} and \ref{0007},
	it is enough to show that
	$R^{S}(\mathcal{A}[p]) \cong H^{1}(X, \mathcal{A}[p]^{\natural, S})$ is infinite.
	Let $U' \subset U$ be the locus where $\mathcal{A}[p]^{\natural}$ is \'etale
	and set $S' = X \setminus U'$.
	Consider the long exact sequence of fppf cohomology groups with compact support
		\[
				\cdots
			\to
				H^{i}_{c}(U', \mathcal{A}[p]^{\natural})
			\to
				H^{i}(X, \mathcal{A}[p]^{\natural, S})
			\to
				\bigoplus_{v \in S'}
					H^{i}(\Order_{v}, \mathcal{A}[p]^{\natural, S})
			\to
				\cdots
		\]
	(\cite[Chapter III, Proposition 0.4 (c), Remark 0.6 (b)]{Mil06},
	\cite[Propositions 2.7.4 and 2.7.5]{Suz20a},
	\cite[Proposition 2.1 (3)]{DH19}).
	The \'etaleness of $\mathcal{A}[p]^{\natural}$ over $U'$ implies that
	$H^{i}_{c}(U', \mathcal{A}[p]^{\natural})$ is finite for all $i$
	by \cite[Chapter VI, Theorem 2.1]{Mil80}.
	Hence it is enough to show that
		$
				H^{1}(\Order_{v}, \mathcal{A}[p]^{\natural, S})
			\cong
				H^{1}(\Order_{v}, \mathcal{A}[p]^{\natural})
		$
	is infinite for a supersingular place $v$ ($\in S' \setminus S$).
	Recall again that $\mathcal{A}[p]^{\natural}$ over $\Order_{v}$
	is finite flat with \'etale generic fiber and non-\'etale special fiber.
	Now for any finite flat group scheme $N$ over $\Order_{v}$
	with \'etale generic fiber,
	the group $H^{1}(\Order_{v}, N)$ is finite
	if and only if it is zero
	if and only if $N$ is \'etale over $\Order_{v}$
	by \cite[Propositions 3.7 and 3.2 (2)]{OS23}.
	Hence $H^{1}(\Order_{v}, \mathcal{A}[p]^{\natural})$ is infinite.

        One can explicitly calculate the $\mu$-invariant of $R^{S}(A / K_{\infty})$
        in terms of the $\Order_{v}$-length of
        $e^{\ast} \Omega^{1}_{\mathcal{A}[p^{n}]^{\natural} / \Order_{v}}$
        for supersingular $v \in S'$ and large enough $n$
        using \cite[Proposition 2.4.2]{LLSTT21} and
        \cite[Propositions 3.7]{OS23},
        where $e$ denotes the zero section
        for $\mathcal{A}[p^{n}]^{\natural} / \Order_{v}$.
\end{example}

\begin{example} \label{0043}
	In this example,
	we will see that the inclusions
	$0 \subset R^{S}(A / K) \subset \Sel(A / K)$
	are both strict even up to finite groups
	and for sufficiently large $S$.
	We will also calculate the characteristic ideals of their Pontryagin duals
	when $k$ is finite.
	Let $X$ and $A$ be ordinary elliptic curves over $k$.
	Let $S$ be any non-empty finite set of places of $K$.
	View $A$ also as $A \times_{k} X$ or $A \times_{k} K$.
	
	First assume that $k = \closure{k}$.
	Note that the connected-\'etale sequence
	$0 \to A[p^{\infty}]^{0} \to A[p^{\infty}] \to \pi_{0}(A[p^{\infty}]) \to 0$
	canonically splits (since it is over $k$).
	Therefore $A[p^{\infty}]^{\natural} \cong \pi_{0}(A[p^{\infty}])$
	and hence $R^{S}(A / K) \cong H^{1}(X, \pi_{0}(A[p^{\infty}]))$.
	We have
		\begin{align*}
					H^{1}(X, \pi_{0}(A[p^{\infty}]))
			&	\cong
					H^{1}(X, \Q_{p} / \Z_{p}) \tensor_{\Z_{p}} T_{p} \pi_{0}(A[p^{\infty}])
			\\
			&	\cong
						(T_{p} \pi_{0}(X[p^{\infty}]))^{\vee}
					\tensor_{\Z_{p}}
						T_{p} \pi_{0}(A[p^{\infty}]),
		\end{align*}
	where $\vee$ denotes the Pontryagin dual.
	Hence
		\[
				R^{S}(A / K)^{\vee}
			\cong
				\Hom_{\Z_{p}} \bigl(
					T_{p} \pi_{0}(A[p^{\infty}]),
					T_{p} \pi_{0}(X[p^{\infty}])
				\bigr),
		\]
	which has rank $1$ over $\Z_{p}$,
	while a similar calculation with $\Sel(A / K) \cong H^{1}(X, A[p^{\infty}])$
	shows that
		\begin{align*}
					\Sel(A / K)^{\vee}
				\cong
			&			\Hom_{\Z_{p}} \bigl(
							T_{p} \pi_{0}(A[p^{\infty}]),
							T_{p} \pi_{0}(X[p^{\infty}])
						\bigr)
			\\
			&	\quad
					\oplus
						\Hom_{\Z_{p}} \bigl(
							T_{p} \pi_{0}(X[p^{\infty}]),
							T_{p} \pi_{0}(A[p^{\infty}])
						\bigr),
		\end{align*}
	which has rank $2$.
        Note that the natural injection from
        $A(K) \tensor \Q_{p} / \Z_{p} \cong
        \Hom_{k}(X, A) \tensor \Q_{p} / \Z_{p}
        \cong \Hom_{k}(A, X) \tensor \Q_{p} / \Z_{p}$
        into $\Sel(A / K)$
        corresponds to the diagonal in the dual of this presentation.
	
	Next assume that $k$ is finite.
	The sequence
	$0 \to A[p^{\infty}]^{0} \to A[p^{\infty}] \to \pi_{0}(A[p^{\infty}]) \to 0$
	still canonically splits.
	We have
		\[
				R^{S}(A / K \closure{k})^{\vee}
			\cong
				\Hom_{\Z_{p}} \bigl(
					T_{p} \pi_{0}(A[p^{\infty}]),
					T_{p} \pi_{0}(X[p^{\infty}])
				\bigr),
		\]
		\begin{align*}
					\Sel(A / K \closure{k})^{\vee}
				\cong
			&			\Hom_{\Z_{p}} \bigl(
							T_{p} \pi_{0}(A[p^{\infty}]),
							T_{p} \pi_{0}(X[p^{\infty}])
						\bigr)
			\\
			&	\quad
					\oplus
						\Hom_{\Z_{p}} \bigl(
							T_{p} \pi_{0}(X[p^{\infty}]),
							T_{p} \pi_{0}(A[p^{\infty}])
						\bigr)
		\end{align*}
	as $\Z_{p}[[\Gal(\closure{k} / k)]]$-modules.
	Let $\alpha_{X}, \alpha_{A} \in \Z_{p}^{\times}$ be
	the unit roots of the Weil polynomials for $X$ and $A$, respectively.
	Then, if $\alpha_{X} \not\equiv \alpha_{A} \mod p$,
	then $R^{S}(A / K_{\infty})^{\vee}$ and $\Sel(A / K_{\infty})^{\vee}$ are
	both zero.
	If $\alpha_{X} \equiv \alpha_{A} \mod p$, then
		\[
				R^{S}(A / K_{\infty})^{\vee}
			\cong
				\Hom_{\Z_{p}} \bigl(
					T_{p} \pi_{0}(A[p^{\infty}]),
					T_{p} \pi_{0}(X[p^{\infty}])
				\bigr),
		\]
		\begin{align*}
					\Sel(A / K_{\infty})^{\vee}
				\cong
			&			\Hom_{\Z_{p}} \bigl(
							T_{p} \pi_{0}(A[p^{\infty}]),
							T_{p} \pi_{0}(X[p^{\infty}])
						\bigr)
			\\
			&	\quad
					\oplus
						\Hom_{\Z_{p}} \bigl(
							T_{p} \pi_{0}(X[p^{\infty}]),
							T_{p} \pi_{0}(A[p^{\infty}])
						\bigr)
		\end{align*}
	as $\Z_{p}[[\Gal(K_{\infty} / K)]] \cong \Z_{p}[[T]]$-modules.
	Therefore the characteristic ideals of
	$R^{S}(A / K_{\infty})^{\vee}$ and $\Sel(A / K_{\infty})^{\vee}$ are
		\[
				(1 - \alpha_{X} / \alpha_{A} + T),
			\quad
				(1 - \alpha_{X} / \alpha_{A} + T)
				(1 - \alpha_{A} / \alpha_{X} + T),
		\]
	respectively.
	
	From the above calculations,
	one can also see that the intersection of
	$R^{S}(A / K_{\infty})$ and $A(K_{\infty}) \tensor \Q_{p} / \Z_{p}$
	in $\Sel(A / K_{\infty})$ is zero.
	Hence $R^{S}(A / K_{\infty})$ contains no contribution
	from the Mordell--Weil group $A(K_{\infty})$.
\end{example}

\begin{remark}
	Let $S$ be large enough.
	The intersection of
	$R^{S}(A / K_{\infty})$ and $A(K_{\infty}) \tensor \Q_{p} / \Z_{p}$
	in $\Sel(A / K_{\infty})$ is called the fine Mordell--Weil group (\cite[Section 2]{Wut07}).
	This subgroup seems to be finite in many cases,
	even when $R^{S}(A / K_{\infty})$ is infinite,
	contrary to the number field case treated in \cite{Wut07}.
	
	To see examples of this finiteness, let $\dim(A) = 1$.
	If $A$ is non-isotrivial, then $R^{S}(A / K_{\infty})$ is finite
	by Corollary \ref{0040} \eqref{0041} and the paragraph thereafter.
	Since $R^{S}(A / K_{\infty})$ is of cofinite type in general by Theorem \ref{0014} \eqref{0015},
	we may ignore bounded torsion and work up to isogeny.
	If $A$ is isotrivial, then we may therefore assume that $A$ is constant.
	If $A$ is supersingular, then $R^{S}(A / K_{\infty}) = 0$
	by Corollary \ref{0040} \eqref{0042}.
	If $A$ and $X$ are both ordinary elliptic curves over $k$,
	then $R^{S}(A / K_{\infty})$ is non-zero in general
	but the fine Mordell--Weil group is zero
	as we saw at the end of Example \ref{0043}.
	
	The case with $\dim(A) > 1$ is less clear.
	But Corollary \ref{0040} \eqref{0041} and \eqref{0042} and the paragraph thereafter
	basically say that
	$R^{S}(A / K_{\infty})$ is finite or even zero
	if $A$ is ``very general'' (for example, ordinary with maximal Kodaira-Spencer rank)
	or  ``very special'' (for example, has infinitesimal $A[p]$).
\end{remark}


\section{%
	\texorpdfstring{Ramified $p$-adic Lie extensions}
	{Ramified p-adic Lie extensions}
}
\label{0051}

Below we freely use the theory of (not necessarily commutative) $p$-adic Lie groups,
their completed group algebras and finitely generated modules over them.
References for this theory include
\cite{Laz65}, \cite{LM87} and \cite{How02}.

Let $k = \closure{k}$, $X$, $K$, $A$, $S$, $U$ and $\mathcal{A}$ be
as in Sections \ref{0053} and \ref{0025}.
Let $L / K$ be a (possibly infinite) Galois extension.
The fine Selmer group $R^{S}(A / L)$ of $A \times_{K} L$ can be defined as before,
which is the direct limit of $R^{S}(A / L')$ over finite subextensions $L'$ of $L / K$.
Its Pontryagin dual $R^{S}(A / L)^{\vee}$ is naturally
a module over $\Z_{p}[[\Gal(L / K)]]$.

\begin{theorem} \label{0048}
	Assume that
	\begin{itemize}
		\item
			$\Gal(L / K)$ is a pro-$p$ $p$-adic Lie group without elements of order $p$,
		\item
			$L / K$ is unramified outside a finite set of places of $K$ and
		\item
			$S \ne \emptyset$ and $\mathcal{A}[p]^{\natural}$ is \'etale over $U$.
	\end{itemize}
	Then the following are true:
	\begin{enumerate}
		\item \label{0045}
			$R^{S}(A / L)^{\vee}$ is finitely generated over $\Z_{p}[[\Gal(L / K)]]$.
		\item \label{0046}
			For any finite set of places $S'$ of $K$ containing $S$,
			we have $R^{S'}(A / L) = R^{S}(A / L)$.
		\item \label{0047}
			Let $n \ge 0$.
			If $A(K^{\sep})[p^{\infty}]$ is killed by $p^{n}$,
			then $R^{S}(A / L)$ is also killed by $p^{n}$.
	\end{enumerate}
\end{theorem}

Note that the statement requires no relation
between the places ramifying in $L / K$ and the set $S$.

\begin{proof}
	\eqref{0045}
	Let $G = \Gal(L / K)$.
	By the profinite Nakayama lemma
	\cite[Corollary in Section 3]{BH97},
	it is enough to show that $R^{S}(A / L)^{G}$ is of cofinite type over $\Z_{p}$.
	
	Fix a non-empty finite set of places $T$ of $K$
	such that $L$ is unramified outside $T$.
	Let $X_{L}$ be the normalization of $X$ in $L$.
	Set $V = X \setminus T$.
	Let $T_{L}$ and $V_{L}$ be the inverse images of $T$ and $V$,
	respectively, in $X_{L}$.
	Let $g_{L} \colon \Spec L_{\et} \to X_{L, \et}$ be the natural morphism
	on the \'etale sites.
	For a finite Galois subextension $L'$ of $L / K$,
	we use similar notation $X_{L'}$, $T_{L'}$, $V_{L'}$ and $g_{L'}$.
	
	By Proposition \ref{0028} \eqref{0037} and a limit argument
	applied to each finite Galois subextension of $L / K$,
	we have a canonical isomorphism
	$R^{S}(A / L) \cong H^{1}(X_{L}, g_{L, \ast}(A[p^{\infty}]^{\natural}))$
	of $G$-modules.
	For $i \ge 0$, define a group
	$H^{i}_{c}(V_{L}, g_{L, \ast}(A[p^{\infty}]^{\natural}))$
	to be the direct limit of the usual compact support \'etale cohomology
	$H^{i}_{c}(V_{L'}, g_{L', \ast}(A[p^{\infty}]^{\natural}))$
	with respect to the contravariant functoriality of $H^{i}_{c}$
	in finite (and hence proper) morphisms,
	where $L'$ runs over finite Galois subextensions of $L / K$.
	The long exact sequence of compact support \'etale cohomology
	gives an exact sequence
		\[
				0
			\to
				A(L)[p^{\infty}]
			\to
				\bigoplus_{w \in T_{L}}
					A(L_{w})[p^{\infty}]
			\to
				H^{1}_{c}(V_{L}, g_{L, \ast}(A[p^{\infty}]^{\natural}))
			\to
				R^{S}(A / L)
			\to
				0,
		\]
	where $L_{w}$ is the (henselian or complete) local field of $L$ at $w$.
	Let $C_{L}$ be the image of the middle morphism
		\[
				\bigoplus_{w \in T_{L}}
					A(L_{w})[p^{\infty}]
			\to
				H^{1}_{c}(V_{L}, g_{L, \ast}(A[p^{\infty}]^{\natural})).
		\]
	The Hochschild--Serre spectral sequence
		\[
				E_{2}^{i j}
			=
				H^{i} \bigl(
					G,
					H^{j}_{c}(V_{L}, g_{L, \ast}(A[p^{\infty}]^{\natural}))
				\bigr)
			\Longrightarrow
				H^{i + j}_{c}(V, g_{\ast}(A[p^{\infty}]^{\natural}))
		\]
	and $H^{0}_{c}(V_{L}, g_{L, \ast}(A[p^{\infty}]^{\natural})) = 0$ show that
		\[
				H^{1}_{c}(V_{L}, g_{L, \ast}(A[p^{\infty}]^{\natural}))^{G}
			\cong
				H^{1}_{c}(V, g_{\ast}(A[p^{\infty}]^{\natural})).
		\]
	Therefore we have an exact sequence
		\[
				H^{1}_{c}(V, g_{\ast}(A[p^{\infty}]^{\natural}))
			\to
				R^{S}(A / L)^{G}
			\to
				H^{1}(G, C_{L}).
		\]
	The similar exact sequence
		\[
				0
			\to
				A(K)[p^{\infty}]
			\to
				\bigoplus_{v \in T}
					A(K_{v})[p^{\infty}]
			\to
				H^{1}_{c}(V, g_{\ast}(A[p^{\infty}]^{\natural}))
			\to
				R^{S}(A / K)
			\to
				0
		\]
	and Proposition \ref{0028} \eqref{0037} show that
	$H^{1}_{c}(V, g_{\ast}(A[p^{\infty}]^{\natural}))$ is of cofinite type over $\Z_{p}$.
	Therefore it is enough to show that
	$H^{1}(G, C_{L})$ is of cofinite type over $\Z_{p}$.
	We have an exact sequence
		\[
				H^{1} \left(
					G,
					\bigoplus_{w \in T_{L}}
						A(L_{w})[p^{\infty}]
				\right)
			\to
				H^{1}(G, C_{L})
			\to
				H^{2}(G, A(L)[p^{\infty}]).
		\]
	We have
		\[
				H^{1} \left(
					G,
					\bigoplus_{w \in T_{L}}
						A(L_{w})[p^{\infty}]
				\right)
			\cong
				\bigoplus_{v \in T}
				H^{1} \left(
					G_{w},
					A(L_{w})[p^{\infty}]
				\right),
		\]
	where, on the right-hand side,
	we chose an element $w \in T_{L}$ lying over each $v \in T$
	and $G_{w} \subset G$ is the decomposition group at $w$.
	By assumption, $G_{w}$ is also a pro-$p$ $p$-adic Lie group
	without elements of order $p$.
	Hence the groups
		\[
				H^{i}(G, A(L)[p^{\infty}]),
			\quad
				H^{i}(G_{w}, A(L_{w})[p^{\infty}])
		\]
	are of cofinite type over $\Z_{p}$ for all $i$
	by \cite[Theorem 1.1]{How02}.
	
	\eqref{0046}, \eqref{0047}
	These follow from Propositions \ref{0028} and \ref{0012}
	applied to each finite Galois subextension of $L / K$
	and a limit argument.
\end{proof}

We now assume the setting of Section \ref{0049},
so $k$ is finite.
Let $L / K$ be a Galois extension.
Again we have the fine Selmer group $R^{S}(A / L)$ of $A \times_{K} L$,
which is the direct limit of $R^{S}(A / L')$ over finite subextensions $L'$ of $L / K$.
Its Pontryagin dual $R^{S}(A / L)^{\vee}$ is naturally
a module over $\Z_{p}[[\Gal(L / K)]]$.

\begin{theorem} \label{0050}
	Assume that
	\begin{itemize}
		\item
			$L$ contains $K_{\infty}$,
		\item
			$\Gal(L / K)$ is a pro-$p$ $p$-adic Lie group without elements of order $p$,
		\item
			$L / K$ is unramified outside a finite set of places of $K$ and
		\item
			$S \ne \emptyset$ and $\mathcal{A}[p]^{\natural}$ is \'etale over $U$.
	\end{itemize}
	Then $R^{S}(A / L)^{\vee}$ is a finitely generated torsion $\Z_{p}[[\Gal(L / K)]]$-module
	whose $\mu$-invariant is zero.
\end{theorem}

Here, recall from the first paragraph of \cite[Section 2]{CS05}
that a finitely generated torsion $\Lambda$-module $M$
(where $\Lambda = \Z_{p}[[\Gal(L / K)]]$) is said to be
\emph{pseudo-null} if $\Ext^{1}_{\Lambda}(M, \Lambda) = 0$
and that $M$ is said to have \emph{$\mu$-invariant zero}
if $M[p^{\infty}]$ is pseudo-null.

\begin{proof}
	By Theorem \ref{0048} \eqref{0045} applied to $K \closure{k}$,
	we know that $R^{S}(A / L \closure{k})^{\vee}$ is finitely generated over
	$\Z_{p}[[\Gal(L \closure{k} / K \closure{k})]]$.
	By an argument similar to the proof of Proposition \ref{0052},
	this implies that $R^{S}(A / L)^{\vee}$ is finitely generated over
	$\Z_{p}[[\Gal(L / K_{\infty})]]$.
	From this, the statement of the proposition follows
	by the general remarks in the proof of \cite[Corollary 3.3]{CS05}.
\end{proof}

\begin{remark} \mbox{}
	\begin{enumerate}
	\item
		In the number field setting,
		Coates--Sujatha also make another, related conjecture
		\cite[Section 4, Conjecture B]{CS05}
		about pseudo-nullity of fine Selmer groups over $p$-adic Lie extensions.
		A function field version of this conjecture is
		studied by Ghosh--Jha--Shekhar \cite{GJS22},
		who obtain a positive result \cite[Theorem 3.14]{GJS22}
		for certain $\Z_{p}^{2}$-extensions $L / K$
		and certain ordinary elliptic curves $A$.
		Here are what our Theorems \ref{0048} and \ref{0050} can say
		about the pseudo-nullity of $R^{S}(A / L)^{\vee}$:
		
		Under the assumptions of Theorem \ref{0050},
		if furthermore $A(K^{\sep})[p^{\infty}]$ is finite,
		then $R^{S}(A / L)^{\vee}$ is killed by a power
		of $p$ by Theorem \ref{0048} \eqref{0047}.
		With the vanishing of $\mu$ (Theorem \ref{0050}),
		this implies that $R^{S}(A / L)^{\vee}$
		is indeed pseudo-null. Note that this includes $A$
		being any non-isotrivial elliptic curve
		as we noted after Corollary \ref{0040}.
		
		If moreover $A(K^{\sep})[p^{\infty}]$ is zero, 
		then $R^{S}(A / L)^{\vee}$ is zero by the same
		reason. Note again that this includes $A$
		being any elliptic curve with $j$-invariant
		not a $p$-th power in $K$ or $A$ being any
		supersingular elliptic curve over $K$.
	\item
		Venjakob \cite[Definition 3.32]{Ven02} treats
		compact $p$-adic Lie groups $G$ without $p$-torsion,
		not necessarily pro-$p$ but with $\F_{p}[[G]]$ integral,
		to define the $\mu$-invariant
		for finitely generated $\Z_{p}[[G]]$-modules.
		In contrast, we assumed in this section that
		$\Gal(L / K)$ is pro-$p$ for simplicity
		since the cited references \cite{CS05} and \cite{How02} make this assumption.
		Also the profinite Nakayama lemma is difficult to use
		in the non-pro-$p$ case.
		More work will be needed to extend our results to non-pro-$p$ extensions.
	\end{enumerate}
\end{remark}


\newcommand{\etalchar}[1]{$^{#1}$}

\end{document}